\documentclass[a4paper]{amsart} 
\usepackage{amssymb} 
\usepackage{verbatim}
\usepackage[all]{xy}
\SelectTips{cm}{}

\usepackage{hyperref}
\calclayout

\title{Abelian length categories of strongly unbounded type}

\thanks{Version from January 28, 2013.}

\author{Henning Krause}
\address{Fakult\"at f\"ur Mathematik\\
Universit\"at Bielefeld\\ D-33501 Bielefeld\\ Germany}
\email{hkrause@math.uni-bielefeld.de}

\newtheorem{lem}{Lemma}[section]
\newtheorem{prop}[lem]{Proposition}
\newtheorem{cor}[lem]{Corollary}
\newtheorem{thm}[lem]{Theorem}
\newtheorem{conj}[lem]{Conjecture}

\theoremstyle{remark}
\newtheorem{rem}[lem]{Remark}

\theoremstyle{definition}

\newtheorem{defn}[lem]{Definition}
\newtheorem{question}[lem]{Question}

\numberwithin{equation}{section}

\renewcommand{\mod}{\operatorname{\mathsf{mod}}\nolimits}
\DeclareMathOperator*{\colim}{\mathsf{colim}}

\newcommand{\Ob}{\operatorname{\mathsf{Ob}}\nolimits}

\newcommand{\Sp}{\operatorname{\mathsf{Sp}}\nolimits}

\newcommand{\rad}{\operatorname{\mathsf{rad}}\nolimits}

\newcommand{\End}{\operatorname{\mathsf{End}}\nolimits}
\newcommand{\Hom}{\operatorname{\mathsf{Hom}}\nolimits}

\renewcommand{\deg}{\operatorname{\mathsf{deg}}\nolimits}

\newcommand{\Coker}{\operatorname{\mathsf{Coker}}\nolimits}

\newcommand{\sub}{\operatorname{\mathsf{sub}}\nolimits}

\newcommand{\Fp}{\operatorname{\mathsf{Fp}}\nolimits}
\newcommand{\Lex}{\operatorname{\mathsf{Lex}}\nolimits}

\newcommand{\length}{\operatorname{\mathsf{length}}\nolimits}
\newcommand{\lend}{\operatorname{\mathsf{endol}}\nolimits}

\newcommand{\Ab}{\operatorname{\mathsf{Ab}}\nolimits}
\newcommand{\op}{\mathrm{op}}

\newcommand{\comp}{\mathop{\circ}}
\newcommand{\lto}{\longrightarrow}
\newcommand{\xto}{\xrightarrow}

\def\a{\alpha}
\def\b{\beta}

\def\p{\phi}

\def\A{{\mathsf A}}

\def\C{{\mathsf C}}
\def\D{{\mathsf D}}

\def\M{{\mathsf M}}

\def\U{{\mathsf U}}
\def\V{{\mathsf V}}

\def\bbN{{\mathbb N}}

\begin{document}

\maketitle
%\setcounter{tocdepth}{1}
%\tableofcontents

\begin{abstract}
  We discuss the notion of strongly unbounded type for abelian length
  categories; this is closely related to the Second Brauer--Thrall
  Conjecture for artin algebras. A new ingredient is the space of
  characters in the sense of Crawley-Boevey.
\end{abstract}

\section{Introduction}

The Second Brauer--Thrall Conjecture says that over infinite fields
every finite dimensional algebra of infinite representation type is of
strongly unbounded type \cite{Ja1957}. In this note we take a fresh
look at the conjecture. We consider abelian length categories which
are Hom-finite over some commutative ring and discuss the analogue of
`stronly unbounded type' in this slightly more general setting. We use
Crawley-Boevey's theory of characters \cite{CB1992,CB1994a,CB1994b}
and further ideas from \cite{He1997,Kr1998}. 

A character for an abelian length category $\C$ is an integer valued
function $\Ob\C\to\bbN$ which is compatible with the exact structure
(Definition~\ref{de:char}).  It turns out that the collection $\Sp\C$
of irreducible characters carries a rich structure. It is a
topological space via Ziegler's topology \cite{Zi1984}, and we
investigate its basic properties. For instance, the Second
Brauer--Thrall Conjecture amounts to the fact that $\Sp\C$ is discrete
if and only if $\C$ has only finitely many isomorphism classes of
indecomposable objects (Theorem~\ref{th:main}). A crucial property is
the compactness of $\Sp\C$. There is an elegant proof of the Second
Brauer--Thrall Conjecture when $\Sp\C$ is compact. On the other hand,
we point out an obstruction for the compactness of $\Sp\C$ in terms of
certain minimal closed subsets of $\Sp\C$
(Proposition~\ref{pr:sub}). This is illustrated by looking at modules
over hereditary artin algebras (Corolloray~\ref{co:hered}).  The
closed subsets of $\Sp\C$ we look at correspond to subcategories of
$\C$ which are closed under subobjects; their relevance became
apparent in recent work \cite{KrPr2012,Ri2012}. Here, we show that
each irreducible character is determined by such a closed subset
(Corollary~\ref{co:poset}).

At this stage,  not much seems to be known about the space
$\Sp\C$. So we end this note with a list of open problems and hope to
stimulate  further progress.

\section{Artin algebras}

The following conjecture is the modification of the \emph{Second
  Brauer--Thrall Conjecture} suggested by Crawley-Boevey in
\cite[\S1.6]{CB1991}. Recall that the \emph{endolength} of a module
$M$ is the length of $M$ when viewed as a module over its endomorphism
ring.

\begin{conj}[Brauer--Thrall II]\label{co:BT}
  Let $A$ be an artin algebra of infinite representation type. Then for
  some $n\in\bbN$ there are infinitely many non-isomorphic
  indecomposable $A$-modules which are of endolength $n$ and of finite
  length over $A$.
\end{conj}

Without going into details, let us mention that this conjecture has
been established for artin algebras over a field $k$ when $k$ is
algebraically closed \cite{Ba1985,Bo1985}, and more generally when $k$
is perfect \cite{BS2005}.

\section{Abelian length categories}
 
Let $k$ be a commutative ring and $\C$ a $k$-linear abelian length
category such that each morphism set in $\C$ has finite length as a
$k$-module. Suppose also that $\C$ has only finitely many
non-isomorphic simple objects.

\begin{defn}[Crawley-Boevey \cite{CB1994b}]\label{de:char}
  A \emph{character} for $\C$ is a function $\chi\colon\Ob\C\to\bbN$
  satisfying the following: \begin{enumerate}
\item $\chi(X\oplus Y)=\chi(X)+\chi(Y)$ for all $X,Y\in\Ob\C$, and 
\item $\chi(X)+\chi(Z)\ge\chi(Y)$ for each exact sequence $X\to Y\to Z\to 0$ in $\C$.
\end{enumerate}
A character $\chi\neq 0$ is \emph{irreducible} if $\chi$ cannot be
written as a sum of two non-zero characters. Note that any character
can be written in a unique way as a finite sum of irreducible
characters \cite{CB1994a}. The \emph{degree} of a character $\chi$ is
\[\deg\chi=\sum_{S}\chi(S)\]
where $S$ runs through a representative set of simple objects in $\C$.
\end{defn}

We denote by $\Sp\C$  the set of irreducible characters for $\C$.
Fix a morphism $\a\colon X\to Y$ in $\C$ and $n\in\bbN$. For each character $\chi$
set \[\chi(\a)=\chi(X)-\chi(Y)+\chi(\Coker\a)\] and let
\begin{equation*}
\U_\a=\{\chi\in\Sp\C\mid \chi(\a)\neq 0\}\quad\text{and}\quad
\V_{\a,n}=\{\chi\in\Sp\C\mid\chi(\a)\le n\}.
\end{equation*}

The following lemma describes a topology on $\Sp\C$; it is the
analogue of Ziegler's topology on the isomorphism classes of
indecomposable pure-injective modules over a ring \cite{Zi1984}.

\begin{lem}\label{le:top}
  The subsets of the form $\U_\a$ where $\a$ runs through all
  morphisms in $\C$ form a basis of open subsets for a topology on
  $\Sp\C$. The subsets of the form $\V_{\a,n}$ are closed. Moreover,
\begin{equation*}
\V_{n}=\{\chi\in\Sp\C\mid\deg\chi\le n\}
\end{equation*}
is closed and compact.
\end{lem}
\begin{proof}
  Consider the abelian category $\A=\Fp(\C,\Ab)$ of additive functors
  $F\colon\C\to\Ab$ into the category $\Ab$ of abelian groups that
  admit a \emph{presentation} \[\Hom(Y,-)\lto\Hom(X,-)\lto F\lto 0.\]
  A function $\chi\colon\Ob\A\to\bbN$ is called \emph{additive}
  provided that $\chi(F)=\chi(F')+\chi(F'')$ if $0\to F'\to F\to
  F''\to 0$ is an exact sequence. Restricting a function
  $\chi\colon\Ob\A\to\bbN$ to $\Ob\C$ by setting
  $\chi(X)=\chi(\Hom(X,-))$ gives a bijection between the additive
  functions $\Ob\A\to\bbN$ and the characters $\Ob\C\to\bbN$.  The
  inverse map takes a character $\chi$ for $\C$ to the function
  $\Ob\A\to\bbN$ which again we denote by $\chi$ and which is defined
  by $\chi(F)=\chi(\a)$ when $F$ is presented by a morphism $\a\colon
  X\to Y$ in $\C$; this does not depend on the choice of $\a$ by
  Schanuel's Lemma.

  Now the assertions follow from properties of additive functions on
  $\A$ established in Lemmas~B.5, B.6, and C.1 in \cite{Kr2012}.  For the
  compactness of $\V_n$ one uses  that $\C$ has only finitely many
  simple objects.
\end{proof}

\begin{rem}
  To each irreducible character $\chi\in\Sp\C$ corresponds a simple
  object $S_\chi$ in some abelian quotient category of $\A$. The
  endomorphism ring of $S_\chi$ is a division ring; it is an
  interesting invariant of $\chi$ (see \cite[Remark~B.3]{Kr2012}) but
  not relevant here.
\end{rem}

The \emph{character} and the \emph{degree} of an object $M$ in $\C$
are defined
by \[\chi_M(-)=\length_{\End(M)}\Hom(-,M)\qquad\text{and}\qquad \deg
M=\deg\chi_M.\] Note that $\chi_M$ is irreducible when $M$ is
indecomposable, by \cite[Theorem~3.6]{CB1994b}. A character of the
form $\chi_M$ with $M$ in $\C$ is called \emph{finite}.

\begin{thm}\label{th:main}
  Let $\C$ be a $k$-linear Hom-finite abelian category having only
  finitely many non-isomorphic simple objects. Consider the following
  statements.
\begin{enumerate}
\item The number of isomorphism classes of indecomposable objects in $\C$ is finite.  
\item For every $n\in\bbN$ the number of non-isomorphic indecomposable
  objects in $\C$ which are of degree $n$ is finite.
\item Every irreducible character for $\C$ is finite.
\item The space of irreducible characters is discrete.
\end{enumerate} 
Then \emph{(2)--(4)} are equivalent. Assuming in addition that
Conjecture~\ref{co:BT} holds, they are also equivalent to \emph{(1)}.
\end{thm}

We need some preparations for the proof. Let $\A=\Lex(\C^\op,\Ab)$ be
the category of left exact functors $\C^\op\to\Ab$. This is a
Grothendieck abelian category and $\C$ identifies via the Yoneda
embedding taking $X$ in $\C$ to $\Hom(-,X)$ with the full subcategory
of finite-length objects in $\A$; see \cite[Chap.~II]{Ga1962} for
details. 

As before and following \cite{CB1994b}, we assign to each object $M$
in $\A$ its character $\chi_M$ (a function
$\Ob\C\to\bbN\cup\{\infty\}$), and $M$ is by definition
\emph{endofinite} if the values of $\chi_M$ are finite.

\begin{thm}[Crawley-Boevey]
The assignment $M\mapsto \chi_M$
induces a bijection between the isomorphism classes of indecomposable
endofinite objects in $\A$ and the irreducible
characters for $\C$.
\end{thm}
\begin{proof} 
This is Theorem~3.6 in \cite{CB1994b};
see also Proposition~B.2 in \cite{Kr2012}.
\end{proof}

The next lemma shows that the finite irreducible characters form a
dense subset of $\Sp\C$.

\begin{lem}\label{le:cl}
Let $\chi_M$ be an irreducible character for $\C$ where $M$ is an object
in $\A$. Write $M=\colim_i M_i$ as a filtered colimit of
objects in $\C$ and let $\U$ be the set of irreducible characters of
the form $\chi_N$ with $N$ a direct summand of some $M_i$. Then
$\chi_M$ belongs to the closure of $\U$.
\end{lem}
\begin{proof}
  Let $\a\colon X\to Y$ be a morphism in $\C$ and $F\colon\A\to\Ab$
  the corresponding functor with
  presentation \[\Hom(Y,-)\lto\Hom(X,-)\lto F\lto 0.\] Observe that
  $\chi_M\in \U_\a$ iff $F(M)\neq 0$. We have $\colim_i F(M_i)\xto{\sim}
  F(M)$, and therefore $\chi_M\in\U_\a$ implies $\chi_N\in \U_\a$ for
  some $\chi_N\in\U$. Thus $\chi_M$ belongs to the closure of $\U$.
\end{proof}

\begin{cor}\label{co:top}
The finite irreducible characters  form a dense subset of $\Sp\C$.\qed
\end{cor}

The following lemma identifies the isolated points in $\Sp\C$. Recall
that a morphism $\p\colon M\to N$ in $\C$ is \emph{left almost split} if $\p$ is
not a split monomorphism and every morphism $M\to N'$ in $\C$ which
is not a split monomorphism factors through $\p$.

\begin{lem}\label{le:isolated}
  Let $\chi\in\Sp\C$. Then $\{\chi\}$ is open if and only if there is
  a left almost split morphism $M\to N$ in $\C$ with $\chi=\chi_M$.
\end{lem}
\begin{proof}
  Suppose first that $\{\chi\}=\U_\a$ for some morphism $\a\colon X\to
  Y$ in $\C$. It follows from Corollary~\ref{co:top} that
  $\chi=\chi_M$ for some indecomposable object $M$ in $\C$.  Choose a
  morphism $\b\colon X\to M$ which does not factor through $\a$ and
  form the following pushout.
\[\xymatrix{ X\ar[r]^\a\ar[d]_\b&Y\ar[d]\\ M\ar[r]&N
}\]
It is easily checked that the morphism $M\to N$ is left almost split.
Conversely, if $\a\colon M\to N$ is left almost split in $\C$, then
$\U_\a=\{\chi_M\}$.
\end{proof}

\begin{proof}[Proof of Theorem~\ref{th:main}]
  There are two cases. Suppose first that there is a simple object $S$
  that does not embed into an injective object in $\C$. Let $E$ be an
  injective envelope of $S$ in $\Lex(\C^\op,\Ab)$ and write
  $E=\bigcup_i E_i$ as the directed union of non-zero objects in
  $\C$. Note that $E$ has degree one. Each $E_i$ has a simple socle
  and is therefore indecomposable of degree one as well. On the other
  hand, the lengths of the $E_i$ are unbounded. The object $E$ yields
  an irreducible character $\chi_E$ for $\C$ which lies in the closure
  of the $\chi_{E_i}$, by Lemma~\ref{le:cl}.  Thus the space $\Sp\C$
  is indiscrete. It follows that each of (1)--(4) does not hold.

  Now suppose that $\C$ has enough injective objects. Let $E$ be an
  injective cogenerator and $A$ its endomorphism ring. Then $A$ is an
  artin algebra and the functor $\Hom(-,E)$ identifies $\C$ with the
  category of finite-length $A$-modules. Note that each indecomposable
  object in $\C$ is the source of a left almost split morphism.

  (3) $\Leftrightarrow$ (4): The finite characters in $\Sp\C$ are
  precisely the isolated points, by Lemma~\ref{le:isolated}. Thus
  $\Sp\C$ is discrete iff all irreducible characters are finite.

  (3) $\Rightarrow$ (2): The set $\V_n$ of irreducible characters of
  degree at most $n$ form a closed subset of $\Sp\C$ which is compact,
  by Lemma~\ref{le:top}. Every character in $\V_n$ is finite and
  therefore isolated, by Lemma~\ref{le:isolated}. Thus $\V_n$ is a
  finite set.

  (2) $\Rightarrow$ (3): Suppose there is an irreducible character
  $\chi$ which is not finite. Using the correspondence between
  irreducible characters for $\C$ and indecomposable endofinite
  $A$-modules \cite[\S5.3]{CB1992}, it follows that $\chi$ corresponds
  to an endofinite $A$-module which is not of finite length. This
  implies that there are infinitely many non-isomorphic indecomposable
  finite-length $A$-modules of some fixed endolength; see
  \cite[\S9.6]{CB1992}. For each $A$-module $M$ we have
\begin{equation*}\label{eq:deg}
\deg M\le\chi_M(A/\rad A)\le\chi_M(A)=\lend M.
\end{equation*}
Thus for some $n\in\bbN$ the
number of non-isomorphic indecomposable objects in $\C$ which are of
degree $n$ is infinite.

(1) $\Rightarrow$ (2): Clear.

(2) $\Rightarrow$ (1): This follows from Conjecture~\ref{co:BT}, using that
the degree of an $A$-module is bounded by its endolength.
\end{proof}

\begin{rem}
  The original notion of `strongly unbounded type' asks for infinitely
  many $n\in\bbN$ such that the number of non-isomorphic
  indecomposable objects of length $n$ is infinite \cite{Ja1957}. We
  cannot expect to have infinitely many such $n$ when \emph{length} is
  replaced by \emph{degree} as in Theorem~\ref{th:main}. To see this,
  consider the category of finitely generated torsion modules over any
  discrete valuation domain: all indecomposables are of degree one.
\end{rem}

\section{Subobject-closed subcategories}

In this section we fix a $k$-linear Hom-finite abelian category $\C$.
%having only finitely many non-isomorphic simple objects.
A full additive subcategory $\D\subseteq\C$ is called
\emph{subobject-closed} if each subobject of an object in $\D$ belongs
to $\D$. Note that in this case the inclusion $\D\to\C$ has a left
adjoint; it takes an object $X\in\C$ to $X/U$ where $U\subseteq X$ is
the minimal subobject with $X/U\in\D$.  Thus $\D$ is a category having
cokernels and we can define characters $\Ob\D\to\bbN$ as before.

\begin{lem}\label{le:sub}
  Let $\D\subseteq\C$ be a subobject-closed subcategory and $p\colon
  \C\to\D$ be  the left adjoint of the inclusion. The map
  $\Sp\D\to\Sp\C$ sending $\chi$ to $\chi\comp p$ identifies $\Sp\D$
  with the closure of $\{\chi_M\in\Sp\C\mid M\in\D\}$.
\end{lem}
\begin{proof}
Adapt the proof of Proposition~2.2 in \cite{KrPr2012}.
\end{proof}

A subset $\U\subseteq\Sp\C$ is said to be \emph{subobject-closed} if
$\U=\Sp\D$ for some subobject-closed additive subcategory
$\D\subseteq\C$.

The subobject-closed subcategories of $\C$ form a set which is
partially ordered by inclusion.  An intersection $\bigcap_\a\C_\a$ of
subobject-closed additive subcategories $\C_\a\subseteq\C$ is again
subobject-closed. Moreover, adapting the proof of
\cite[Corollary~2.3]{KrPr2012}, we have
\[\Sp(\bigcap_\a\C_\a)=\bigcap_\a\Sp\C_\a.\]
Thus we can define for each $\chi\in\Sp\C$ the
subobject-closed subset
\[\sub\chi=\bigcap_{\chi\in\U}\U\]
where $\U$ runs through all subobject-closed subsets
$\U\subseteq\Sp\C$.

\begin{thm}\label{th:sub}
Let $\chi,\psi$ be irreducible characters for $\C$. Then
\[\sub\chi=\sub\psi\quad\iff\quad\chi=\psi.\]
\end{thm}

The proof is based on the following lemma.

\begin{lem}\label{le:sub}
Let $M$ be an indecomposable endofinite object in $\A=\Lex(\C^\op,\Ab)$
and $\p\colon M\to \coprod_I M$ a monomorphism. Then at least one
component of $\p$ is invertible.
\end{lem}
\begin{proof}
Consider in $\A$ the colimit of the chain of monomorphisms
\[M\stackrel{ \p }\lto\coprod_{I}M\xto{\coprod_I\p}\coprod_{I^2}M
\xto{\coprod_I\coprod_I\p}\coprod_{I^3}M
\xto{\coprod_I\coprod_I\coprod_I\p}\cdots\] and write this as morphism
$\bar\p\colon M\to \coprod_J M$. Denote by $E$ the endomorphism ring
of $M$; it is a local ring with $\bigcap_{n\ge 0}\rad E=0$. This
follows from \cite[Prop.~IV.14]{Ga1962} since $M$ can be identified
with an indecomposable injective object of a locally finite abelian
category; see \cite[Proposition~B.2]{Kr2012}. If each component of
$\p$ belongs to $\rad E$, then each component of $\bar\p$ belongs
$\bigcap_{n\ge 0}\rad E$. Thus there is one component of $\p$ which is
invertible.
\end{proof}

\begin{proof}[Proof of Theorem~\ref{th:sub}]
  As before, we work in the category $\A= \Lex(\C^\op,\Ab)$. For
  $M\in\A$ let $\sub M$ denote the subobject-closed subcategory of
  $\C$ consisting of all finite-length subobjects of finite coproducts
  of copies of $M$. Note that $M$ is a filtered colimit of objects in
  $\sub\M$.  

  Now fix indecomposable endofinite objects $M$ and $N$ in
  $\A$. Then \[\sub\chi_M=\Sp(\sub M)\] and therefore
\[\sub\chi_M\subseteq\sub\chi_N\quad\iff\quad\sub M\subseteq\sub N.\]
Suppose that $\sub\chi_M\subseteq\sub\chi_N$. Write $M=\colim_i M_i$
as a filtered colimit of finite-length objects and choose for each $i$
a monomorphism $\a_i\colon M_i\to N_i$ into a finite coproduct of
copies of $N$. Note that the coproducts of copies of $N$ form a
subcategory of $\A$ which is closed under filtered colimits and
products, since $N$ is endofinite; see \cite[\S3]{CB1994b} or
\cite[Corollary~10.5]{Kr1998b}.  Thus the $\a_i$ induce a
monomorphism \[M=\colim_i M_i\lto\colim_i(\prod_{i\to j}N_j)=\coprod_I
N\] for some index set $I$.

Now suppose that $\sub\chi_M=\sub\chi_N$. Thus there is also a
monomorphism $N\to \coprod_J M$ for some index set $J$. Composing
these morphisms yields a monomorphism \[\p\colon M\lto
\coprod_I\coprod_J M.\] Each component of $\p$ is an
endomorphism of $M$ that factors through a coproduct of copies of $N$,
and it follows from Lemma~\ref{le:sub} that at least one component is
invertible.  Thus $M\cong N$ and therefore $\chi_M=\chi_N$.
\end{proof}

\begin{cor}\label{co:poset}
The set $\Sp\C$ is partially ordered by 
\[\chi\subseteq \psi\quad\iff\quad \sub\chi\subseteq\sub\psi.\]
In particular, each  $\chi\in\Sp\C$ is uniquely
determined by the set of finite characters in $\sub\chi$.
\end{cor}
\begin{proof}
  The first part is clear from Theorem~\ref{th:sub}. For the second
  part, observe that $\sub\chi$ is the closure of the subset of finite
  characters in $\sub\chi$ by Lemma~\ref{le:sub}.
\end{proof}

Let us rephrase Theorem~\ref{th:sub} for module categories, using the
correspondence between characters and endofinite modules
\cite[\S5.3]{CB1992}. For a module $M$, let $\sub M$ denote the
category consisting of all finite-length submodules of finite
direct sums of copies of $M$.

\begin{cor}
  Two indecomposable endofinite modules $M$ and $N$ over an artin
  algebra are isomorphic if and only if $\sub M=\sub N$.\qed
\end{cor}

\section{Compactness}

We keep the setting of the previous section and discuss the
compactness of $\Sp\C$. More specifically, the fact that each
irreducible character is determined by a subobject-closed subset of
$\Sp\C$ raises the question when a subset of $\Sp\C$ is of the form
$\sub\chi$ for some character $\chi$. We have the following criterion.

\begin{prop}\label{pr:sub}
  Suppose that each indecomposable object in $\C$ is the source of a
  left almost split morphism.  Let $\U\subseteq\Sp\C$ be an infinite
  subobject-closed subset which is minimal with respect to this
  property. Then the following conditions are equivalent.
\begin{enumerate}
\item $\U$ is compact.
\item $\U$ contains an infinite character.
\item $\U$ is of the form $\sub\chi$ for some  character $\chi$.
\end{enumerate}
\end{prop}

We need the following lemma.

\begin{lem}\label{le:closure}
  Let $\U\subseteq \Sp\C$ and $\chi\in\Sp\C$. Choose indecomposable
  endofinite objects $(M_i)_{i\in I}$ and $M$ in $\A=\Lex(\C^\op,\Ab)$
  such that $\U=\{\chi_{M_i}\mid i\in I\}$ and $\chi=\chi_M$. Then the
  following conditions are equivalent.
\begin{enumerate}
\item The character $\chi$ belongs to the closure of $\U$.
\item The object $M$ belongs to the smallest subcategory of $\A$ closed under
  products, filtered colimits, pure subobjects, and 
containing all $M_i$.
\end{enumerate}
\end{lem}
\begin{proof}
This follows from Corollary~4.6 and Theorem~6.2 in \cite{Kr1998b}.
\end{proof}

\begin{proof}[Proof of Proposition~\ref{pr:sub}]
  As before, we work in the category $\A= \Lex(\C^\op,\Ab)$. We denote
  by $\D$ the subobject-closed subcategory of $\C$ such that $\U=\Sp\D$.

(1) $\Rightarrow$ (2): Each finite character in $\Sp\C$ is isolated,
by Lemma~\ref{le:isolated}. It follows that $\U$ contains an infinite
character when $\U$ is infinite and compact.

(2) $\Rightarrow$ (3): Let $\chi\in\U$ be infinite. Then
$\sub\chi\subseteq\U$, and the minimality of $\U$ implies equality.

(3) $\Rightarrow$ (1): Let $\U=\sub\chi$ with $\chi=\chi_M$ for some
indecomposable endofinite object $M$ in $\A$.  It suffices to show
that each infinite closed subset $\V\subseteq\U$ contains $\chi$. Thus
we choose such an infinite subset $\V$. The minimality of $\U$ implies
that the smallest subobject-closed subset of $\Sp\C$ containing $\V$
coincides with $\U$. Write $M=\colim_i M_i$ as a filtered colimit of
objects in $\D$. It follows that for each $i$ there is a monomorphisms
$M_i\to N_i$ with $N_i$ a finite coproduct of indecomposable objects
$N\in\C$ such that $\chi_N\in\V$. There are also monomorphisms $N_i\to
P_i$ such that each $P_i$ is a finite coproduct of copies of $M$,
since $N_i\in\D=\sub M$.  These monomorphisms induce a pair of
monomorphisms \[M=\colim_i M_i\lto\colim_i(\prod_{i\to
  j}N_j)\lto\colim_i(\prod_{i\to j}P_j)=\coprod_I M\] for some index
set $I$, as in the proof of Theorem~\ref{th:sub}. It follows from
Lemma~\ref{le:sub} that at least one component of this composite is
invertible. Thus $\chi$ belongs to $\V$, by Lemma~\ref{le:closure}. We
conclude that $\U$ is compact.
\end{proof}

There is an interesting consequence for the space of irreducible
characters on the category of finite-length modules over a hereditary
artin algebra. I am grateful to Claus Michael Ringel for providing the
key observation.

\begin{cor}\label{co:hered}
  Let $A$ be a connected hereditary artin algebra and $\mod A$ be the
  category of finite-length $A$-modules. Then the space $\Sp(\mod A)$
  is compact if and only if $A$ is of finite or tame representation
  type.
\end{cor}
\begin{proof}
  The assertion is clear from Theorem~\ref{th:main} when $A$ is of
  finite representation type, and it follows by inspection of the
  Ziegler spectrum when $A$ is tame, see \cite{Pr1998, Ri1998}. Now
  suppose that $A$ is of wild representation type and denote by $\C$
  the category of preprojective modules; it is minimal among the
  subobject-closed subcategories of $\mod A$ which are of infinite
  type, by \cite[Example~1]{Ri2012}. There is no indecomposable
  endofinite $A$-module of infinite length which is a union of modules
  from $\C$, by \cite{Ri2013}.  Thus the corresponding subset $\Sp\C$
  is not compact, by Proposition~\ref{pr:sub}. It follows that
  $\Sp(\mod A)$ is not compact.
\end{proof}

\section{Some open Problems}

Not much seems to be known about the space $\Sp\C$ of characters for a
length category $\C$. For instance, we need a better understanding of
the interplay between finite and infinite characters.  In this section
we address some open problems. We fix a $k$-linear Hom-finite abelian
category $\C$ having only finitely many non-isomorphic simple objects.

\subsection*{The support of a character}

It would be interesting to find for each irreducible character $\chi$
some appropriate set of finite irreducible characters
\emph{supporting} $\chi$. More specifically, we ask the following.

\begin{question}
Is every irreducible character of degree $n$
in the closure  of a set of finite irreducible characters of
  degree at most $n$?
\end{question}

\subsection*{Subobject-closed subsets}

The lattice of subobject-closed subsets of $\Sp\C$ contains a lot of
information. For instance, the space $\Sp\C$ itself embeds naturally into
this lattice, by Theorem~\ref{th:sub}. This motivates the following
question.

\begin{question}
When is a subobject-closed subsets of $\Sp\C$ of the form $\sub\chi$
for some irreducible character $\chi$?
\end{question}

There is a compactness result for subobject-closed subcategories which
is due to Ringel \cite{Ri2012}; his proof uses properties of the
Gabriel--Roiter measure. An alternative proof in \cite{KrPr2012} is
based on the compactness of the Ziegler spectrum.

A full additive subcategory $\D\subseteq\C$ is of \emph{infinite type}
if there are infinitely many isomorphism classes of indecomposable
objects in $\D$.

\begin{prop}
  Suppose that each indecomposable object in $\C$ is the source of a
  left almost split morphism.  Then each subobject-closed additive
  subcategory of $\C$ that is of infinite type contains one which is
  minimal among all subobject-closed additive subcategories of
  infinite type.
\end{prop}
\begin{proof}
Adapt the proof of Corollary~4.3 in  \cite{KrPr2012}.
\end{proof}

From this proposition it follows that each infinite subobject-closed
subset of $\Sp\C$ contains a minimal one.

\begin{question}
  When does a minimal infinite subobject-closed subset of $\Sp\C$
  contain an infinite character?
\end{question}

Note that such a minimal infinite subobject-closed subset with an
infinite character $\chi$ is necessarily of the form $\sub\chi$, by
Proposition~\ref{pr:sub}.

\begin{rem}
  Let $\D\subseteq\C$ be a minimal subobject-closed subcategory of
  infinite type. Then for each $n\in\bbN$ the number of isomorphism
  classes of indecomposable objects of length $n$ in $\D$ is finite;
  see \cite[Theorem~2]{Ri2012}.
\end{rem} 

\subsection*{Compactness}

When $\Sp\C$ is compact, one gets an elegant proof of the Second
Brauer--Thrall Conjecture, because an infinite space cannot be
discrete; see Theorem~\ref{th:main}. This observation motivates the
following question.

\begin{question}
When is $\Sp\C$ compact?
\end{question}

Even when $\Sp\C$ is not compact, we can ask for suitable subsets
which are compact. Here is one possible question.

\begin{question}
When is a subset of the form $\sub\chi$ compact?
\end{question}

%Note that such minimal infinite subobject-closed subset 

\subsection*{Partial orders on the set of characters}

The set of characters $\Ob\C\to\bbN$ is partially ordered by
\[\chi\le \psi\quad\iff\quad \chi(X)\le\psi(X)\text{  for all
}X\in\Ob\C.\]
Observe that for each character $\psi$ the set
\[\U_\psi=\{\chi\in\Sp\C\mid \chi\le \psi\}\]
is closed and compact; this follows from Lemma~\ref{le:top}.

From Theorem~\ref{th:sub} it follows that the inclusion relation for
$\sub\chi$ provides another partial order on the set of irreducible
characters.

\begin{question}
How are the partial orders $\le$ and $\subseteq$ on $\Sp\C$ related?
\end{question}
We refer to \cite{KrPr2012,Sch2004} for more details on these partial orders.


\begin{thebibliography}{99}
%
\bibitem{Ba1985} R. Bautista, On algebras of strongly unbounded representation
type. Comment. Math. Helv. {\bf 60} (1985), no.~3, 392--399.
%
\bibitem{BS2005} R. Bautista\ and\ L. Salmer\'on, On discrete and inductive algebras, in {\it Representations of algebras and related topics}, 17--35, Fields Inst. Commun., 45 Amer. Math. Soc., Providence, RI, 2005.
%
\bibitem{Bo1985} K. Bongartz, Indecomposables are standard,
  Comment. Math. Helv. {\bf 60} (1985), no.~3, 400--410.
%
\bibitem{CB1991} W. Crawley-Boevey, Tame algebras and generic modules, Proc. London Math. Soc. (3) {\bf 63} (1991), no.~2, 241--265.
%
\bibitem{CB1992} W. Crawley-Boevey, Modules of finite length over their endomorphism
rings, in {\it Representations of algebras and related topics (Kyoto,
  1990)}, 127--184, London Math. Soc. Lecture Note Ser., 168 Cambridge
Univ. Press, Cambridge, 1992.
%
\bibitem{CB1994a} W. Crawley-Boevey, Additive functions on locally
  finitely presented Grothendieck categories, Comm. Algebra {\bf 22}
  (1994), no.~5, 1629--1639.
%
\bibitem{CB1994b} W. Crawley-Boevey, Locally finitely presented
  additive categories, Comm.  Algebra {\bf 22} (1994), no.~5,
  1641--1674.
% 
\bibitem{Ga1962}  P. Gabriel, Des cat\'egories ab\'eliennes,
  Bull. Soc. Math. France {\bf 90} (1962), 323--448.
%
\bibitem{He1997} I. Herzog, The Ziegler spectrum of a locally coherent
  Grothendieck category, Proc. London Math. Soc. (3) {\bf 74} (1997) 503--558.
%
\bibitem{Ja1957} J. P. Jans, On the indecomposable representations of
  algebras, Ann. of Math. (2) {\bf 66} (1957), 418--429.
%
\bibitem{Kr1998} H. Krause, Generic modules over Artin algebras, Proc. London
Math. Soc. (3) {\bf 76} (1998), no.~2, 276--306.
%
\bibitem{Kr1998b} H. Krause, Exactly definable categories, J. Algebra
  {\bf 201} (1998), no.~2, 456--492.
%
\bibitem{Kr2012} H. Krause, Cohomological length functions, arXiv:1209.0540.
%
\bibitem{KrPr2012} H. Krause and M. Prest, The Gabriel-Roiter
  filtration of the Ziegler spectrum, Quart. J. Math.,
  doi:10.1093/qmath/has020.
%
\bibitem{Pr1998} M. Prest, Ziegler spectra of tame hereditary
  algebras, J. Algebra {\bf 207} (1998), no.~1, 146--164.
%
\bibitem{Ri1998} C. M. Ringel, The Ziegler spectrum of a tame hereditary algebra,
Colloq. Math. {\bf 76} (1998), no.~1, 105--115.
%
\bibitem{Ri2012} C. M. Ringel, Minimal infinite submodule-closed subcategories,
Bull. Sci. Math. {\bf 136} (2012), no.~7, 820--830.
%
\bibitem{Ri2013} C. M. Ringel, Generic modules over wild hereditary
  artin algebras, in preparation.
%
\bibitem{Sch2004} K. Schmidt, The endofinite spectrum of a tame
  algebra, J. Algebra {\bf 279} (2004), no.~2, 771--790.
%
\bibitem{Zi1984} M. Ziegler, Model theory of modules, Ann. Pure
  Appl. Logic {\bf 26} (1984), no.~2, 149--213.
%
\end{thebibliography}
\end{document}